\documentclass[11pt]{amsart}  
\usepackage{amssymb}  
\usepackage[utf8]{inputenc}
\usepackage[english]{babel}
\usepackage{lineno,hyperref}
\usepackage{amsmath}
\usepackage{amscd}
\usepackage{amsthm}
\usepackage{amsfonts}
\usepackage{color}
\usepackage{xfrac}
\usepackage{mathtools}
\usepackage{graphicx}
\usepackage{diagrams} 
\usepackage{tikz}
\usetikzlibrary{arrows}
\usepackage{bbold}
\usepackage[T1]{fontenc}
\newtheorem{theorem}{Theorem}[section]
\newtheorem{proposition}[theorem]{Proposition}
\newtheorem{corollary}[theorem]{Corollary}
\newtheorem{lemma}[theorem]{Lemma}

\newtheorem{remark}[theorem]{Remark}

\numberwithin{equation}{section}
\newcommand\restr[2]{\ensuremath{\left.#1\right|_{#2}}}

\AtBeginDocument{\setlength{\linenumbersep}{\dimexpr\oddsidemargin+0.2in-\linenumberwidth\relax}} 
\begin{document}  
\title[Generalized manifolds, normal invariants, and $\mathbb{L}$-homology]{ 
Generalized manifolds, normal invariants, and $\mathbb{L}$-homology} 
\author[F. Hegenbarth and D. Repov\v{s}]{ 
Friedrich Hegenbarth and Du\v{s}an Repov\v{s}}
\address{Dipartimento di Matematica "Federigo Enriques", 
Universit\` a degli studi di Milano, 
20133 Milano, Italy}
\email{friedrich.hegenbarth@unimi.it} 
\address{Faculty of Education and Faculty of Mathematics and Physics, University of Ljubljana \&
Institute of Mathematics, Physics and Mechanics, 
1000 Ljubljana, Slovenia}
\email{dusan.repovs@guest.arnes.si}
\begin{abstract}  
Let  $X^{n}$ be an arbitrary
 oriented closed generalized $n$-ma\-ni\-fold,
  $n\ge 5$.
 In our recent paper (Proc. Edinb. Math. Soc. (2) 63 (2020), no. 2, 597-607) we have constructed
 a map 
 $t:\mathcal{N}(X^{n}) \to H^{st}_{n} ( X^{n}; \mathbb{L}^+)$
 which extends the normal invariant map for the case when 
 $X^{n}$
 is a topological $n$-manifold.
 Here, 
 $\mathcal{N}(X^{n})$ denotes the set of all normal bordism classes of degree one normal maps 
 $(f,b): M^{n} \to X^{n},$
 and
 $H^{st}_{*} ( X^{n}; \mathbb{E})$
 denotes the Steenrod homology of the spectrum $\mathbb{E}$. 
 An important nontrivial question arose whether
 the
  map
 $t$ is bijective (note that this holds in the case that $X^{n}$ is a topological $n$-manifold).
 It is the purpose of this paper to prove that the answer to this question is affirmative. 
\end{abstract} 
\dedicatory{Dedicated to the memory of Professor Erik Kj{\ae}r Pedersen (1946-2020)} 
\maketitle
\section{Introduction}\label{s1} 
Throughout the paper, $n$ will denote an integer $\ge 5$.  
A {\it generalized ma\-ni\-fold  $X^{n}$  (without boundary)}
of dimension $n\in \mathbb{N}$  is a 
{\it Euclidean neighborhood retract (ENR)}
 (i.e. $X^{n}$ is an $n$-dimensional locally compact separable metri\-za\-ble absolute neighborhood retract (ANR)), 
 satisfying the {\sl local Poincar\'{e} duality}
(i.e. the local homology of $X^{n}$ is like that of $\mathbb{R}^n$).

In this paper we shall consider only {\it oriented connected  compact} gene\-ra\-lized manifolds.
Clearly, every oriented {\it closed} (i.e. connected, compact and without boundary) topological manifold is such a space
(cf. Cavicchioli, Hegen\-barth and Repov\v{s}~\cite{CHR16}).

 For every generalized $n$-manifold $X^{n},$ 
 there exists an embedding 
 $\varphi: X^{n}\hookrightarrow \mathbb{R}^{m}$
  into $\mathbb{R}^m$, 
 for a sufficiently large $m\ge n \in \mathbb{N}$,
 so that the boundary $\partial N^{m} \subset \mathbb{R}^{m}$ of a neighbourhood 
 $N^{m} \subset \mathbb{R}^m$ of $\varphi(X^{n})$ in $\mathbb{R}^m$
 is homotopy equivalent to a spherical fibration 
 $\nu_{X^{n}}$, 
 called the {\it Spivak fibration}, 
 with fiber homotopy equivalent to $S^{m-n-1}$
 (cf.  Browder~\cite{Br72}).
 We shall consider only the oriented case and we
 shall denote
 also  its classifying map by
 $\nu_{X^{n}}:X^{n}\to BSG.$
 
 A systematic construction of generalized manifolds was given by
 Bryant, Ferry, Mio and Weinberger~\cite {BFMW96}
 (for a comprehensive treatment see Cavicchioli, Hegenbarth and Repov\v{s}~\cite{CHR16}
 and
 Hegenbarth and Repov\v{s}~\cite{HeRe06}, 
 and the references therein).
 It was proved by Ferry and Pedersen~\cite{FePe95} that there is a
 canonical lift
 $\xi_{0}:X^{n} \to BSTOP$ of $\nu_{X^{n}}$, i.e.
 the composition
 $X^{n} \xrightarrow[]{\xi_{0}} BSTOP \xrightarrow[]{\mathcal{J}} BSG$
 is homotopic to  $\nu_{X^{n}}$.
 It gives rise to the {\it canonical surgery problem,}
   denoted by
 $(f_{0}, b_{0})$, via the Pontryagin-Thom construction.
 
 Here, 
 $f_{0}:M^{n}_{0} \to X^{n}$ 
 is a degree one map, where 
 $M^{n}_{0}$
 is a closed topological $n$-manifold
 and 
 $b_{0}:\nu_{M^{n}_{0}} \to \xi_{0}$
 is a bundle map, covering the map $f_{0}$ (by slightly abusing the notation, we shall 
  denote
  by  $\nu_{M_{0}^{n}}$
  also
   the
 stable normal  $\mathbb{R}^{m-n}$-bundle of 
 an embedding
 $M^{n}_{0} \hookrightarrow  \mathbb{R}^{m}$, not just its associated spherical fibration).
 The canonical surgery problem $(f_{0}, b_{0})$ is unique up to normal bordism.
 
 Let us denote the set of all normal bordism classes of normal degree one maps $(f,b)$ by $\mathcal{N}(X^{n})$,
 where 
 $f:M^{n} \to X^{n}$
 is a map of degree one,
 $b:\nu_{M^{n}} \to \xi$
 is a bundle map covering $f$,
 and
 $\xi:X^{n} \to BSTOP$ is a $TOP$-reduction of
$\nu_{X^{n}}$ 
(i.e. 
$\mathcal{J} \circ \xi$
is homotopic to 
$\nu_{X^{n}})$.

In the case when $X^{n}$ is a closed $n$-manifold, one associates to $(f,b)$
and element in
$H_{n}(X^{n}; \mathbb{L}^{+}),$
where
$\mathbb{L}^{+}=\mathbb{L}<1>$ is the  (semi-simplicial) {\it connected surgery spectrum}
 (cf. 
 K\" uhl, Macko and Mole ~\cite{KMM13},
 Nicas~\cite{Ni82},
 and
 Ranicki~\cite[Chapter 18]{Ra92}).
 
  In the case when $X^{n}$ is a
 topological
  $n$-manifold,
 this element in 
 $H_{n}(X^{n}; \mathbb{L}^{+})$
 is obtained by decomposing 
 $(f,b)$
 into adic pieces, using a transversality structure on the manifold
 $X^{n}$ (cf. Ranicki~\cite[Chapter 16]{Ra92}).
 This defines a map
 $t:\mathcal{N}(X^{n}) \to H_{n}(X^{n}; \mathbb{L}^{+})$
 which is bijective.
 The image of 
 $(f,b)$
 is called the {\it normal invariant} of the normal degree one map
 $(f,b)$.
 
 This construction does not carry over to generalized manifolds
 $X^{n}$.
 If 
 $X^{n}$
 is not homotopy equivalent to a topological $n$-manifold, there is no transversality structure on
 $X^{n}$. Moreover, what does 
 $\mathbb{L}^{+}$-homology 
 mean in the class of  compact ENR's?
 In our recent paper  
  we have proved the following result.  
 
 \begin{theorem}({\hbox{\rm{Hegenbarth-Repov\v{s}}}~\cite[Theorem~5.1]{HeRe20})}\label{A}
 Let $X^{n}$
 be an oriented closed ge\-ne\-ra\-li\-zed $n$-manifold, $n\ge 5$. Then one can construct a map 
 \[
 t:\mathcal{N}(X^{n}) \to H^{st}_{n} ( X^{n}; \mathbb{L}^+)
 \]
 which extends the normal invariant map in the case when 
 $X^{n}$
 is a 
 topological
 $n$-manifold.
 \end{theorem}
 
 Here, 
 $H^{st}_{*} ( X^{n}; \mathbb{E})$
 denotes the {\it Steenrod homology} of the spectrum $\mathbb{E}$.
 We refer to 
 Ferry~\cite{Fe95},
 Kahn, Kaminker and Schochet~\cite{KKS77},
 and Milnor~\cite{Mi95}
 for the construction and properties. 
 
 As it was already pointed out above,  the map 
 $t:\mathcal{N}(X^{n}) \to H^{st}_{n} ( X^{n}; \mathbb{L}^+)$ in Theorem~\ref{A}
 is bijective for topological $n$-manifolds  $X^{n}$. Therefore it is very natural to ask if perhaps bijectivity of $t$ also holds for generalized $n$-manifolds
 $X^{n}$?
 The main goal of the present paper is to show that the answer to
 this question is affirmative.
 
 \begin{theorem}\label{B}
 Let $X^{n}$
 be an oriented closed generalized $n$-manifold, $n\ge 5$. Then the map 
 $
 t:\mathcal{N}(X^{n}) \to H^{st}_{n} ( X^{n}; \mathbb{L}^+)
 $
in Theorem~\ref{A} is also a bijection.
 \end{theorem} 
 
 We outline the plan how we shall prove Theorem~\ref{B}.
 In Section~\ref{s2} we shall  recall the construction of the map 
 $t:\mathcal{N}(X^{n}) \to H^{st}_{n} ( X^{n}; \mathbb{L}^+)$
 from Hegenbarth and Repov\v{s}~\cite{HeRe20}.
In Section~\ref{s3} we shall prove that the map 
$t:\mathcal{N}(X^{n}) \to H^{st}_{n} ( X^{n}; \mathbb{L}^+)$ 
is the composition of maps in the following commutative diagram
\begin{equation}\label{d1}
\begin{tikzpicture}[baseline=-0.8cm, node distance=2.5cm, auto, ]
  \node (Xl) {$\mathcal{N}(M^{n}_{0})$};
  \node (Fl) [right of=Xl] {$H_{n} (M^{n}_{0}; \mathbb{L}^{+})$};
  \node (Xl1) [below of=Xl, yshift=1cm] {$\mathcal{N}(X^{n})$};
  \node (Fl1) [below of=Fl, yshift=1cm] {$H^{st}_{*} ( X^{n}; \mathbb{L}^{+})$};
  \draw[->, font=\small] (Xl) to node [midway, above]{$t_{0}$} (Fl);
  \draw[->, font=\small]  (Xl1) to node {} (Xl);
  \draw[->, font=\small] (Fl) to node {$(f_{0})_{*}$} (Fl1);
  \draw[->, font=\small] (Xl1) to node [midway, above]{$t$} (Fl1);
\end{tikzpicture}
\end{equation} 

 There are canonical identifications of
 $\mathcal{N}(M^{n}_{0})$
 with
 $H^{0} (M^{n}_{0}; \mathbb{L}^{+})$
 and
 $\mathcal{N}(X^{n})$
 with
 $H^{0}(X^{n}; \mathbb{L}^{+})$
 such that
 $\mathcal{N}(X^{n})\to \mathcal{N}(M^{n}_{0})$
 corresponds to
  $$
 (f_{0})^{*}: H^{0}(X^{n}; \mathbb{L}^{+}) \to H^{0} (M^{n}_{0}; \mathbb{L}^{+}).
 $$
 A precise definition will be given at the beginning of Section~\ref{s3}.
 
 Here, 
 $(f_{0}, b_{0})$
 is the canonical surgery problem mentioned above.  
 It is well-known that the composed map 
 \[
 H^{0} (M^{n}_{0}; \mathbb{L}^{+}) 
 \xrightarrow[]{\cong} 
  \mathcal{N}(M^{n}_{0})
   \xrightarrow{t_{0}} 
    H_{n} (M^{n}_{0}; \mathbb{L}^{+})
    \] 
 is equal to the following composition of isomorphisms
 \[
 H^{0} (M^{n}_{0}; \mathbb{L}^{+})
 \xrightarrow[]{\cong}  
 \widetilde{H}^{m-n}(T(\nu_{M^{n}_{0}}); \mathbb{L}^{+})
  \xrightarrow[]{SD}  
 H_{n} (M^{n}_{0}; \mathbb{L}^{+}),
 \]
where $T(\nu_{M^{n}_{0}})$
denotes the {\it Thom space} of the normal bundle of an embedding 
$M^{n}_{0} \hookrightarrow  \mathbb{R}^{m}$
and the first map is the {\it Thom isomorphism}.
The second map $SD$ denotes the $S$-{\it duality} (i.e. the {\it Spanier-Whitehead duality})
isomorphism (cf.
K\" uhl, Macko, and Mole~\cite[Chapter 14,  p.259]{KMM13} 
and 
Ranicki~\cite[Chapter 17, p.193]{Ra92}).

The same isomorphisms hold for $X^{n}$
(cf. Ranicki~\cite[Proposition 16.1 (v), p.175]{Ra92},
\[
H^{0}(X^{n}; \mathbb{L}^{+}) \xrightarrow[]{\cong}  \widetilde{H}^{m-n}(T(\nu_{X^{n}}); \mathbb{L}^{+}),
\]
where we assume 
$X^{n} \hookrightarrow  \mathbb{R}^{m},$
and the existence of the isomorphism
\[
\widetilde{H}^{m-n}(T(\nu_{X^{n}}); \mathbb{L}^{+})
 \xrightarrow[\cong]{SD}  
 H^{st}_{n} (X^{n}; \mathbb{L}^{+})
 \] 
follows from  Kahn, Kaminker and Schochet~\cite[Theorem B, p.205]{KKS77}.

Finally, in Section~\ref{s4} we shall show that since 
 $(f_{0}, b_{0})$
 is a normal degree one map, the following diagram commutes (cf. diagram~\ref{d12} in Section~\ref{s4})    
 \begin{equation}\label{d2}
\begin{tikzpicture}[baseline=-0.8cm, node distance=3.5cm, auto] 
  \node (A) {$H^{0} (M^{n}_{0}; \mathbb{L}^{+})$};
  \node (B) [right of=A] {$\widetilde{H}^{m-n}(T(\nu_{M^{n}_{0}}); \mathbb{L}^{+})$};
   \node (C) [right of=B] {$H_{n} (M^{n}_{0}; \mathbb{L}^{+}) $};  
  \node (E) [below of=A, yshift=2cm] {$H^{0} (X^{n}; \mathbb{L}^{+})$};
  \node (F) [below of=B, yshift=2cm] {$\widetilde{H}^{m-n}(T(\nu_{X^{n}}); \mathbb{L}^{+})$};
   \node (G) [below of=C, yshift=2cm] {$H^{st}_{n} (X^{n}; \mathbb{L}^{+}) $};  
  \draw[->, font=\small] (A) to node {} (B);
  \draw[->, font=\small] (B) to node {$SD$} (C);  
  \draw[->, font=\small] (E) to node {$(f_{0})^{*}$} (A);
  \draw[->, font=\small] (F) to node {$(T(b_{0}))^{*}$} (B);
  \draw[->, font=\small] (C) to node {$(f_{0})_{*}$} (G);  
  \draw[->, font=\small] (E) to node {} (F);  
  \draw[->, font=\small] (F) to node {$SD$} (G);  
\end{tikzpicture}
\end{equation}

The bottom isomorphism is therefore equal to the composite map
\[
H^{0}(X^{n};\mathbb{L}^{+})
\cong
\mathcal{N}(X^{n})
\to
\mathcal{N}(M^{n}_{0})
\xrightarrow[]{t_{0}}
H_{n}(M^{n}_{0};\mathbb{L}^{+})
\xrightarrow[]{(f_{0})_{*}}
H^{st}_{n}(X^{n};\mathbb{L}^{+}).
\]   
Now the commutativity of diagram~\ref{d1}   implies that the map
$
t:\mathcal{N}(X^{n}) \to H^{st}_{n} ( X^{n}; \mathbb{L}^+)
$
is indeed bijective, as asserted in Theorem~\ref{B}.
Details will be given in the forthcoming sections.

\begin{remark}
In the epilogue (cf. Section~\ref{s5}) we shall give an outlook for comparing the exact sequence of a map 
$q:X^{n} \to B,$
where $B$ is a compact metric space, with the controlled surgery sequence, determined by
the map
 $q$
  (cf. Bryant, Ferry, Mio and Weinberger~\cite{BFMW96}).
We are grateful to the referee for suggesting to
also
include a discussion of this interesting
problem.
\end{remark}

\section{Construction of the map $t:\mathcal{N}(X^{n}) \to H^{st}_{n} ( X^{n}; \mathbb{L}^+)$}\label{s2}
We recall the construction of the map $t:\mathcal{N}(X^{n}) \to H^{st}_{n} ( X^{n}; \mathbb{L}^+)$ from
 Hegenbarth and Repov\v{s}~\cite[Section 4]{HeRe20}.
So let us fix an oriented closed  generalized $n$-manifold $X^{n}$ of dimension $n \ge 5$.
If
 $\mathcal{U}$
is a covering of $X^{n}$ by open sets, we denote its {\sl nerve}
by $N(\mathcal{U}).$
If the covering
$\mathcal{U'} \prec \mathcal{U}$ is a refinement of  $\mathcal{U}$,
then
there is a simplicial map
$s:   N(\mathcal{U'}) \to   N(\mathcal{U}).$ 

\begin{proposition}\label{p}
There exists a sequence of open
coverings 
${\{\mathcal{U}_{j} \}}_{j \in \mathbb{N}}$
with the following properties:

(a) for every $j \in \mathbb{N}$,
 $\mathcal{U}_{j+1} \prec \mathcal{U}_{j},$
 and there exists a simplicial map
$s_{j}:N(\mathcal{U}_{j+1}) \to  N(\mathcal{U}_{j});$

(b) for every $j \in \mathbb{N}$,
there exist maps
$\varphi_{j}:X^{n} \to  N(\mathcal{U}_j), \ 
\psi_{j}:N(\mathcal{U}_j) \to X^{n}$
such that 
$\psi_{j} \circ \varphi_{j}: X^{n} \to X^{n}$
is an
$\varepsilon_{j}\rm{-equivalence},$
where
$  \lim_{j \to \infty}{\varepsilon_{j}} = 0;$

(c) there exists a map
$\psi :  \varprojlim_{j} N({\mathcal{U}}_{j}) \to X^n,$
and moreover, if maps 
$s_j : N(\mathcal{U}_{j+1}) \to  N(\mathcal{U}_{j})$
take
$N({\mathcal{U}}_{j+1})$
to a subdivision of
$N({\mathcal{U}}_{j}),$
then 
$\varprojlim_{j} N({\mathcal{U}}_{j})$
can be identified with 
$X^n;$
and

(d) the following diagram is homotopy commutative
\begin{equation}\label{d3}
	\begin{tikzpicture}[baseline=-1cm, node distance=2cm, auto]
	  \node (LU) {};
	  \node (X) [node distance=1cm, below of=LU] {$X^{n}$};
	  \node (NU') [right of=LU] {$N (\mathcal{U}_{j+1})$};
	  \node (NU) [below of=NU'] {$N (\mathcal{U}_{j})$};
	  \node (RU) [right of=NU'] {};
	  \node (X2) [node distance=1cm, below of=RU] {$X^{n}$};
	  \draw[->, font=\small] (X) to node {$\varphi_{j+1}$} (NU');
	  \draw[->, font=\small] (X) to node [swap] {$\varphi_{j}$} (NU);
	  \draw[->, font=\small] (NU') to node {$s_{j}$} (NU);
	  \draw[->, font=\small] (NU') to node {$\psi_{j+1}$} (X2);
	  \draw[->, font=\small] (NU) to node [swap] {$\psi_{j}$} (X2);
	\end{tikzpicture}
	\end{equation}	
\end{proposition}
Hereafter, we shall assume that property (c) holds.

\begin{proof}
See Hegenbarth and Repov\v{s}~\cite[Sections 2 and 3]{HeRe20} 
for verification of  properties (a), (b), (d),
and  Milnor~\cite[Lemma 2]{Mi95} for  property
(c).
\end{proof}
Let 
$
M(s_{j})=
 N (\mathcal{U}_{j+1})
  \times I 
  \underset{s_j}{\cup}
  N (\mathcal{U}_{j})
$
be the mapping cylinder of the map 
$s_{j}:N(\mathcal{U}_{j+1}) \to  N(\mathcal{U}_{j}).$
Using property Proposition~\ref{p} (d), we can form the mapping telescope 
$
F_{0}= \underset{j \in \mathbb{N}}{\cup} M(s_{j})
$
and the obvious maps
\[
X^{n} \times [j, j+1]
\xrightarrow[]{\varphi_{j} \times Id_{[j, j+1]}}
M(s_{j})
\xrightarrow[]{\psi_{j} \times Id_{[j, j+1]}}
X^{n} \times [j, j+1]
\]   
fit together to give the map
$
X^{n}\times \mathbb{R}_{+} 
\xrightarrow[]{\Gamma}
F_{0}
\xrightarrow[]{\Lambda}
X^{n}\times \mathbb{R}_{+}. 
$

Here, $F_{0}$
is a locally finite complex which can be completed to give a complex $F$ such that (cf. Hegenbarth and Repov\v{s}~\cite[Section 3]{HeRe20} 
for details):

(i)  at the $\infty$-end we add
$
\underset{j}{\varprojlim} N({\mathcal{U}}_{j}) = X^{n};
$

(ii) at the $0$-end we add a cone with the cone point $c_{0}.$ 

The complex $F_{0}$ (resp. $F$) 
is an open (resp. closed)
fundamental complex of the (compact metric) space $X^{n}$.
If $\mathbb{E}$ is an arbitrary spectrum
and 
$H^{lf}_{*}(F_{0}; \mathbb{E})$ denotes the locally finite homology
of $F_{0}$, then
the Steenrod
homology sati\-sfies the following axiom
\[
H^{lf}_{*}(F_{0}; \mathbb{E})
\cong
H^{st}_{*}(F, X^{n},\{c_{0}\}; \mathbb{E}).
\]   
Note that $F$ is contractible, hence we have the
following  isomorphism
\[
H^{st}_{m}(F, X^{n},\{c_{0}\}; \mathbb{E})
\xrightarrow[\cong]{\partial}
H^{st}_{m-1}(X^{n}; \mathbb{E}).
\]   

We can now outline the construction of the map 
$t:\mathcal{N}(X^{n}) \to H^{st}_{n} ( X^{n}; \mathbb{L}^+)$
 (cf. Hegenbarth and Repov\v{s}~\cite[Section 4]{HeRe20}).
 Let
 $(f,b)$
 be a normal degree one map, i.e.
 $f:M^{n} \to X^{n}$ is of degree one and 
 $b:\nu_{M^{n}} \to \xi$
 is a bundle map covering $f$. 
 As before, 
 $(f_{0},b_{0})$
 denotes the canonical map, i.e.
 $
 f_{0}:M^{n}_{0} \to X^{n}, \ 
  b_{0}:\nu_{M^{n}_{0}} \to \xi_{0}.
  $
 Consider the following bundles over $F_{0}$:
$
\eta = \Lambda^{*}(\xi \times \mathbb{R}_{+}), \ 
\eta_{0} = \Lambda^{*}(\xi_{0} \times \mathbb{R}_{+}).
$
Then
$
\Gamma^{*}(\eta)\cong \xi \times \mathbb{R}_{+}, \ 
\Gamma^{*}(\eta_{0})\cong \xi_{0} \times \mathbb{R}_{+},
$
since 
$\Lambda \circ \Gamma$
is homotopic to $Id_{X^{n}\times \mathbb{R}_{+} }$.

One obtains bundle maps $(\Phi, B)$ and $(\Phi_{0}, B_{0})$ from the following compositions 
\[
\Phi: M^{n} 
\times
\mathbb{R}_{+} \xrightarrow[]{f \times Id_{\mathbb{R}_{+}}} 
X^{n} \times \mathbb{R}_{+}
\xrightarrow[]{\Gamma}
F_{0},
\]   
\[
B: \nu_{M^{n}} 
\times
\mathbb{R}_{+} \xrightarrow[]{b  \times Id_{\mathbb{R}_{+}}} 
\xi \times \mathbb{R}_{+}
\xrightarrow[]{\Gamma}
\eta,
\]   
\[
\Phi_{0}: M^{n}_{0} 
\times
\mathbb{R}_{+} \xrightarrow[]{f_{0} \times Id_{\mathbb{R}_{+}}} 
X^{n} \times \mathbb{R}_{+}
\xrightarrow[]{\Gamma}
F_{0},
\]   
\[
B_{0}: \nu_{M^{n}_{0}} 
\times
\mathbb{R}_{+} \xrightarrow[]{b_{0} \times  Id_{\mathbb{R}_{+}}} 
\xi_{0} \times \mathbb{R}_{+}
\xrightarrow[]{\Gamma}
\eta_{0}.
\]   

Their mapping cylinders 
 $M(\Phi, B)$ (resp. $M(\Phi_{0}, B_{0})$)
 are
normal spaces with boundaries
$({M^{n}} 
\times
\mathbb{R}_{+})
 \amalg
F_{0}$
(resp.
$({M^{n}_{0}} 
\times
\mathbb{R}_{+})
 \amalg
F_{0}$).
Gluing them along $F_{0}$ 
yields the normal space
\[
N=M(F, B) 
\underset{F_{0}}{\cup}
- M(F_{0}, B_{0}), \ \
 \partial N=
{M^{n}}
\times
\mathbb{R}_{+}
\underset{F_{0}}{\cup}
{M^{n}_{0}} 
\times
\mathbb{R}_{+},
\]
where the minus sign denotes the opposite orientation
on $M(F_{0},B_{0})$.

This normal space $N$ can be decomposed into adic pieces to define an element in
$H^{lf}_{n+2}(F_{0};\mathbb{\Omega}^{NSTOP})$,
 where 
 $\mathbb{\Omega}^{NSTOP}$
 is the semi-simplicially defined spectrum of adic normal spaces with manifold boundary
(cf. K\" uhl, Macko and  Mole~\cite[Section 11]{KMM13}
for the precise definition).

There is a similar spectrum
$\mathbb{\Omega}^{NPD},$
where the boundaries are Poincar\'{e} duality spaces, and
there exists an obvious map
$\mathbb{\Omega}^{NSTOP}
\to
\mathbb{\Omega}^{NPD}.$
Moreover, there is a map of spectra
$
\mathbb{\Omega}^{NPD}
\to
\mathbb{L}^{+}
$
(cf. Ranicki~\cite[p.287]{Ra79}), 
inducing isomorphisms in homology theory
(cf.
Hausmann and Vogel~\cite{HaVo93},
Levine~\cite{Le72},
Quinn~\cite{Qu72}).
The composition
$\mathbb{\Omega}^{NSTOP}
\to
\mathbb{\Omega}^{NPD}
\to
\mathbb{L}^{+}
$
is called $sign^{\mathbb{L}}$
in
K\" uhl, Macko and  Mole~\cite[p.232]{KMM13}.

A word about notations: we shall denote the element represented by 
$M(\Phi, B) 
\underset{F_{0}}{\cup}
-M(\Phi_{0}, B_{0})$
by
$\{f,b\}-\{f_{0},b{_0}\}
\in
H^{lf}_{n+2}(F_{0};\mathbb{\Omega}^{NSTOP})$
and its image under
\[
H^{lf}_{n+2}(F_{0};\mathbb{\Omega}^{NSTOP})
\xrightarrow[]{\cong}
H^{st}_{n+2}(F,X^{n},\{c_{0}\};\mathbb{\Omega}^{NSTOP})
\]   
\[
\xrightarrow[\cong]{\partial}
H^{st}_{n+1}(X^{n};\mathbb{\Omega}^{NSTOP})
\xrightarrow[]{sign^{\mathbb{L}}} 
H^{st}_{n}(X^{n};\mathbb{L}^{+})
\]   
will be denoted by
$[f,b]-[f_{0},b_{0}].$

Finally, one can then  show that the map
$t:\mathcal{N}(X^{n}) \to H^{st}_{n} ( X^{n}; \mathbb{L}^+)$
sending
$(f,b)$ to $[f,b]-[f_{0},b_{0}],$ 
is well-defined (cf. Hegenbarth and Repov\v{s}~\cite[Theorem 5.1]{HeRe20}).

\section{Factorization of the map $t:\mathcal{N}(X^{n}) \to H^{st}_{n} ( X^{n}; \mathbb{L}^+)$}\label{s3}

This section is devoted to studying diagram~\ref{d2}.
\subsection*{I}
First, one has to define the map
$
\mathcal{N}(X^{n})
\to
\mathcal{N}(M^{n}_{0}).
$
We shall keep the notations from Section~\ref{s2}, so 
$(f_{0},b_{0})$
denotes the canonical
surgery problem for an oriented closed  generalized $n$-manifold $X^{n}$
with
$f_{0}:M^{n}_{0} \to X^{n}, \  b_{0}:\nu_{M^{n}_{0}} \to \xi_{0}.$

Let 
$(f,b)$
represent an element in
$
\mathcal{N}(X^{n})$,
where
$ f:M^{n} \to X^{n}, \  b:\nu_{M^{n}} \to \xi.$
We shall also write 
$
\xi_{0}, \xi:X^{n} \to BSTOP
$
 for the corresponding
 classifying maps.
Their compositions with 
$\mathcal{J}: BSTOP \to BSG$
are homotopic.

Consider now the bundles
$(f_{0})^{*}(\xi_{0})$
and
$(f_{0})^{*}(\xi)$
over
$M^{n}_{0}$.
Observe that
 $(f_{0})^{*}(\xi_{0}) = \nu_{M^{n}_{0}}$
 and that 
$(f_{0})^{*}(\xi)$
is fiber homotopy equivalent to 
$ \nu_{M^{n}_{0}}.$
In other words, 
$(f_{0})^{*}(\xi)$
is a $TOP$-reduction of the Spivak fibration of the manifold 
$M^{n}_{0}$.

Therefore
$(f_{0})^{*}(\xi)$
defines a surgery problem
$
f':{M'}^{n} \to M^{n}_{0}, \ 
b':\nu_{{M'}^{n}} \to (f_{0})^{*}(\xi),
$
which we shall denote by
$(f',b').$
These are well-known constructions (cf. 
Browder~\cite[Section II.4]{Br72},
Madsen and Milgram~\cite[Chapter 2]{MaMi79},
Wall~\cite[Chapter 10]{Wa99}).
\begin{lemma}\label{l1}
The composition of the normal maps
$$
{M'}^{n} 
\xrightarrow[]{f'}
M^{n}_{0} 
\xrightarrow[]{f_{0}}
  X^{n} ,\  \
  \nu_{{M'}^{n}} 
\xrightarrow[]{b'}
(f_{0})^{*}(\xi) 
\xrightarrow[]{\tilde{f}_{0}}
  \xi, 
$$
where $\tilde{f}_{0}$ is the obvious bundle map covering the map $f_{0}$, is normally bordant to $(f,b).$
\end{lemma}
\begin{proof}
For the proof we have to describe 
$(f_{0},b_{0}),(f,b),$
and
$(f',b')$
in more details.
Suppose that $X^{n}$ is embedded into  $S^{m},$ for some sufficiently large $m\ge n$, with a regular neighborhood $W^{m}\subset S^{m}$ and a retraction $r:W^{m} \twoheadrightarrow X^{n}$.
Thus 
$
\restr{r}{\partial W^{m}}:
{\partial W^{m}} 
 \twoheadrightarrow
 X^{n}
$ 
is homotopy equivalent to the spherical fibration $\nu_{X^{n}}$, giving rise to
$
\beta: S^{m} \to W^{m}/\partial W^{m}  \to T(\nu_{X^{n}}).
$

The $TOP$-reductions
$\xi_{0}$ and $\xi$ of $\nu_{X^{n}}$ then yield the following homotopy commutative diagram
\begin{equation}\label{d4}
\begin{tikzpicture}[baseline=-1cm,node distance=2cm, auto]
  \node (LU) {};
  \node (X) [node distance=1cm, below of=LU] {$T(\nu_{X^{n}})$};
  \node (NU') [right of=LU] {$T(\xi)$};
  \node (NU) [below of=NU'] {$T(\xi_{0})$};
  \draw[->, font=\small] (X) to node {} (NU');
  \draw[->, font=\small] (X) to node [swap] {} (NU);
  \draw[->, font=\small] (NU) to node {$h$} (NU');
\end{tikzpicture}
\end{equation}

Note that 
$h: T(\xi_{0}) \to T(\xi)$ 
 is induced by a fiber homotopy equivalence
$\dot{\xi}_{0} \sim \nu_{X^{n}} \sim \dot{\xi},$
where
$\dot{\xi}_{0}$
 (resp.
 $\dot{\xi}$) denotes the sphere bundles of 
 $\xi_{0}$
 (resp. $\xi$). 

Denote the compositions with
$\beta$
by
$\alpha_{0}:S^{m} \to T(\xi_{0}), \  \alpha:S^{m} \to T(\xi).$ 
They can be made transverse to 
$X^{n} \subset T(\xi_{0})$
 (resp. $T(\xi)$) in order to obtain
$\alpha^{-1}_{0}(X^{n})=M^{n}_{0}$
 (resp. $\alpha^{-1}(X^{n})=M^{n}$), and $b_{0}$ (resp. $b$) are the obvious maps from their normal bundles in $S^{m}.$
 Moreover, $\alpha_{0}$ (resp. $\alpha$) factor as
 $S^{m} \to T(\nu_{M^{n}_{0}}) \to T(\xi_{0})$
 (resp.  $S^{m} \to T(\nu_{M^{n}}) \to T(\xi)$)
 and we have the following homotopy commutative diagram 
\begin{equation}\label{d5}
\begin{tikzpicture}[baseline=-2.0cm, node distance=3.5cm, auto]
 \node (LU) {};
  \node (X) [node distance=2cm, below of=LU] {$S^{m}$};
  \node (NU') [right of=LU] {$T(\nu_{M^{n}})$};
  \node (NU) [below of=NU'] {$T(\nu_{M^{n}_{0}})$};
  \draw[->, font=\small] (X) to node {} (NU');
  \draw[->, font=\small] (X) to node [swap] {} (NU); 
   \node (C) [right of=B] {$T(\xi) $};  
   \node (G) [below of=C] {$T(\xi_{0}) $}; 
  \draw[shorten <=-0.9cm,->, font=\small] (B) to node {} (C); 
  \draw[->, font=\small] (X) to node {$\alpha$} (C);
  \draw[->, font=\small] (X) to node {$\alpha_{0}$} (G);
  \draw[->, font=\small] (G) to node {$h$} (C); 
  \draw[->, font=\small] (NU) to node {} (G); 
  \end{tikzpicture}
\end{equation}

Note that 
$h: T(\xi_{0}) \to T(\xi)$ 
induces a homotopy equivalence
$
\bar{h}: T((f_{0})^{*}(\xi_{0})) \to T((f_{0})^{*}(\xi)).
$
However,
$(f_{0})^{*}(\xi_{0})=\nu_{M^{n}_{0}},$
so we get the following homotopy commutative diagram
\begin{equation}\label{d6}
\begin{tikzpicture}[baseline=-1.8cm,  ->,>=stealth', auto, node distance=3.5cm, auto]
  \node (A) {$S^{m}$};
  \node (B) [right of=A] {$T(\nu_{M^{n}})$};
   \node (C) [right of=B] {$T(\xi)$}; 
  \node (E) [below of=A] {$T(\nu_{M^{n}_{0}})$};
  \node (F) [below of=B] {$T((f_{0})^{*}(\xi_{0}))$};
   \node (G) [below of=C] {$T((f_{0})^{*}(\xi))$}; 
    \node (X) [node distance=2cm, below of=B] {$T(\xi_{0})$};
     \draw[->, font=\small] (F) to node {$T({\tilde{f}}_{0})$} (X);
      \draw[->, font=\small] (X) to node {$h$} (C);
  \draw[->, font=\small] (A) to node {} (B);
  \draw[->, font=\small] (B) to node {} (C); 
  \draw[->, font=\small] (A) to node {} (E);
  \draw[->, font=\small] (A) to node {$\alpha'$} (E);
  \draw[->, font=\small] (A) to node {$\alpha_{0}$} (X);
  \draw[->, font=\small] (G) to node {$T({\tilde{f}_{0}})$} (C); 
  \draw[->, font=\small] (E) to node {$=$} (F);  
  \draw[->, font=\small] (F) to node {$\bar{h}$} (G);
  \path[every node/.style={font=\sffamily\small}]
   (A) edge[bend left] node [left,above] {$\alpha$} (C); 
\end{tikzpicture}
\end{equation}

Here, 
$\tilde{f}_{0}: (f_{0})^{*}(\xi) \to \xi$
 (resp. $\tilde{f}_{0}: (f_{0})^{*}(\xi_{0}) \to \xi_{0}$)
are the obvious bundle maps 
over 
$f_{0}:M^{n}_{0} \to X^{n}$ (for simplicity we use the same symbol $\tilde{f}_{0}$ for both maps),
and
$T(\tilde{f}_{0})$
is the induced map between the Thom spaces, so 
$T(\tilde{f}_{0})^{-1}(X^{n})=M^{n}_{0},$
similarly for 
$
T(\tilde{f}_{0}):T((f_{0})^{*}(\xi_{0})) \to T(\xi_{0}).
$
Note that $h$ and $\bar{h}$
are not induced by bundle maps.

By making the composition
$
S^{m}  \xrightarrow[]{\alpha'} T(\nu_{M^{n}_{0}}) = T((f_{0})^{*}(\xi_{0})) \xrightarrow[]{\bar{h}} T((f_{0})^{*}(\xi))
$
transverse to $M^{n}_{0},$
one obtains the surgery problem
\[
{M'}^{n}=(\bar{h}\circ {\alpha'})^{-1}(M^{n}_{0}) \to M^{n}_{0}, \ \  b':\nu_{{M'}^{n}} \to (f_{0})^{*}(\xi).
\]   
Homotopy commutativity of diagram~\ref{d6} then implies that
\[
{M'}^{n}
 \xrightarrow[]{f'}  
 M^{n}_{0}  
 \xrightarrow[]{f_{0}}  
 X^{n}, \ \ 
 \nu_{{M'}^{n}} 
 \xrightarrow[]{b'}  
 (f_{0})^{*}(\xi)
  \xrightarrow[]{\tilde{f}_{0}}  
 \xi
\]   
is normally bordant to 
$(f,b).$
To see this,  observe that  $(f,b)$ is obtained from the upper arrow  $\alpha$,
whereas the composition  $(f_{0},b_{0}) \circ {(f',b')}$
is obtained from 
the composition of the arrows 
$\downarrow$$_{\longrightarrow}$$\uparrow$, that is  
$T({\tilde{f}}_{0}) \circ \bar{h} \circ {\alpha'}.$
Note that $T({\tilde{f}}_{0})$  produces $(f_{0},b_{0})$  and  $\bar{h}\circ \alpha'$ gives $(f',b')$.
This completes the proof of Lemma~\ref{l1}. 
\end{proof}

\begin{remark}
One might expect that homotopy comutativity of diagram~\ref{d6}
implies that 
$(f_{0},b_{0})$
and 
$(f,b)$
are normally bordant. However, this is not the case
since $h$ (resp. $\bar{h}$)
are not induced by $TOP$-bundle maps.
\end{remark}
The association
$(f,b) \to (f',b')$
defines a map
$\mathcal{N}(X^{n})
\to 
\mathcal{N}(M^{n}_{0}).
 $
 It depends on the fixed surgery problems
 $(f_{0},b_{0}),$
 and
$
Id_{M^{n}_{0}} : 
 M^{n}_{0} 
 \xrightarrow[]{\cong}
 M^{n}_{0}, \ 
 Id_{\nu_{M^{n}_{0}}}: 
 \nu_{M^{n}_{0}}
 \xrightarrow[]{\cong}
 \nu_{M^{n}_{0}}.
$
 We shall relate this map using the following identifications
 (cf. K\" uhl, Macko and  Mole~\cite[Chapter 14, in particular Section 14.23]{KMM13})
 $$\mathcal{N}(X^{n})
 \to
 [X^{n}, G/TOP], \ \ \mathcal{N}(M^{n}_{0})
 \to
 [M^{n}_{0}, G/TOP].$$
 
Given 
 $
f:{M}^{n} \to X^{n}, \ 
 b:\nu_{{M}^{n}} \to \xi,
$
we know that $\xi \oplus (-\xi_{0}): X^{n} \to BTOP$ classifies the Whitney sum of $\xi$ and $-\xi_{0}$.
The composition with 
$\mathcal{J}:BTOP \to BSG$ is homotopic to the constant map, hence it yields a map
$X^{n} \to G/TOP.$
This defines a bijection
$\mathcal{N}(X^{n}) \to [X^{n},G/TOP],$
depending on $(f_{0},b_{0}).$

Let us denote the image of $(f,b)\in \mathcal{N}(X^{n})$
in
$[X^{n},G/TOP]$
by $[\xi - \xi_{0}].$
Similarly, 
 $\mathcal{N}(M^{n}_{0}) \to [M^{n}_{0},G/TOP]$
 can be defined using  
 $Id_{ M^{n}_{0}}: 
 M^{n}_{0} 
 \xrightarrow[]{\cong}
 M^{n}_{0}, \ 
 Id_{\nu_{M^{n}_{0}}}: 
 \nu_{M^{n}_{0}}
 \xrightarrow[]{\cong}
 \nu_{M^{n}_{0}}.
 $
 The construction above then implies the following corollary.
 
 \begin{corollary}\label{c1}
  The diagram
 \begin{equation}\label{d8}
\begin{tikzpicture}[baseline=-1.0cm, node distance=2.5cm, auto, ]
  \node (Xl) {$\mathcal{N}(M^{n}_{0})$};
  \node (Fl) [right of=Xl] {$[M^{n}_{0},G/TOP]$};
  \node (Xl1) [below of=Xl, yshift=1cm] {$\mathcal{N}(X^{n})$};
  \node (Fl1) [below of=Fl, yshift=1cm] {$ [X^{n},G/TOP]$};
  \draw[->, font=\small] (Xl) to node [midway, above]{} (Fl);
  \draw[->, font=\small]  (Xl1) to node {} (Xl);
  \draw[->, font=\small] (Fl1) to node [swap] {$(f_{0})^{*}$} (Fl);
  \draw[->, font=\small] (Xl1) to node [midway, above]{} (Fl1);
\end{tikzpicture}
\end{equation}
 \end{corollary}
 commutes. Moreover,
 $(f_{0})^{*}([\xi - \xi_{0}])=
 [(f_{0})^{*}(\xi) - \nu_{M^{n}_{0}}].
 $\qed 
 
 \subsection*{II}
 Next, we show how $(f',b')$ 
 can be used to
 calculate
 $t(f,b)\in H^{st}_{n}(X^{n}; \mathbb{L}^{+}).$
 By crossing $(f',b')$
 with
 $\mathbb{R}_{+},$
 one gets a normal map
 \[
f'\times Id_{\mathbb{R}_{+}}:{M'}^{n}\times \mathbb{R}_{+} \to M^{n}_{0}\times \mathbb{R}_{+}, \ \ 
 b'\times Id_{\mathbb{R}_{+}}:\nu_{{M'}^{n}}\times \mathbb{R}_{+} \to (f_{0})^{*}(\xi)\times \mathbb{R}_{+},
\]   
 denoted by
 $(f',b')\times Id_{\mathbb{R}_{+}}.$
 The mapping cylinder
 $M((f',b')\times Id_{\mathbb{R}_{+}})$
 of the map $(f',b')\times Id_{\mathbb{R}_{+}}$
 is a normal space with manifold boundary, hence it defines an element
\[
M((f',b')\times Id_{\mathbb{R}_{+}})\in  H^{lf}_{n+2}(M^{n}_{0}\times \mathbb{R}_{+}; \mathbb{\Omega}^{NSTOP}).
\]   
 
 \begin{lemma}\label{l2}
 Let  $\Gamma_{0}:M^{n}_{0}\times \mathbb{R}_{+}
\to
 F_{0}$ be defined as the composition of the  maps
  $
 f_{0}\times Id_{\mathbb{R}_{+}}: M^{n}_{0}\times \mathbb{R}_{+}
\to
 X^{n}\times \mathbb{R}_{+}$
 and
 $
\Gamma:  X^{n}\times \mathbb{R}_{+}
 \to
 F_{0}.
$
 Then $\Gamma_{0}$ induces a homomorhism
  \[
  (\Gamma_{0})_{*}:  
 H^{lf}_{n+2}(M^{n}_{0}\times \mathbb{R}_{+}; \mathbb{\Omega}^{NSTOP})
 \to
 H^{lf}_{n+2}(F_{0}; \mathbb{\Omega}^{NSTOP}),
 \]   
 such that
 $
 (\Gamma_{0})_{*}([M((f',b')\times Id_{\mathbb{R}_{+}})])=
\{f,b\}-\{f_{0},b_{0}\}.
$  
 \end{lemma}
 
 \begin{proof}
 The element 
  $ (\Gamma_{0})_{*}([M((f',b')\times Id_{\mathbb{R}_{+}})])$
 is represented by
 the mapping cylinder 
\[
 ({M'}^{n}\times \mathbb{R}_{+})\times I 
  \ 
\underset{f'\times Id_{\mathbb{R}_{+}}}{\cup} \
M^{n}_{0}\times \mathbb{R}_{+},
\]   
but decomposed according to the dissection given by
$\Gamma_{0}: M^{n}_{0}\times \mathbb{R}_{+} \to F_{0}.$
The element
$\{f,b\}-\{f_{0},b_{0}\}$
is represented by
\[
({M'}^{n}\times \mathbb{R}_{+})\times I 
 \underset{\Phi}{\cup}
 F_{0}  \ 
\underset{F_{0}}{\bigcup} \
-(M^{n}_{0}\times \mathbb{R}_{+})\times I 
 \underset{\Phi_{0}}{\cup}
 F_{0},
 \]   
as described in Section~\ref{s2}.
By Lemma~\ref{l1}, it is equivalent to the mapping
cylinder construction based on the composition of the normal maps
$(f_{0},b_{0})\circ(f',b').$ It gives
the following
\[
({M'}^{n}\times \mathbb{R}_{+})\times I 
 \underset{f'\times Id_{\mathbb{R}_{+}}}{\cup}
 M^{n}_{0}\times \mathbb{R}_{+}
 \cup
 (M^{n}_{0}\times \mathbb{R}_{+})\times I
 \underset{\Gamma_{0}}{\cup} \
 F_{0}  \
\cup 
-(M^{n}_{0}\times \mathbb{R}_{+})\times I 
 \underset{\Gamma_{0}}{\cup}
 F_{0},
 \]   
(cf. Ferry~\cite[Proposition 8.10]{Fe92} for the mapping cylinder calculations).

This is obviously bordant to 
\[({M'}^{n}\times \mathbb{R}_{+})\times I 
 \underset{f'\times Id_{\mathbb{R}_{+}}}{\cup}
 M^{n}_{0}\times \mathbb{R}_{+}
 \]
since 
\[
  (M^{n}_{0}\times \mathbb{R}_{+})\times I
 \underset{\Gamma_{0}}{\cup} \
 F_{0}  \
\cup 
-(M^{n}_{0}\times \mathbb{R}_{+})\times I 
 \underset{\Gamma_{0}}{\cup}
 F_{0}
\]
 is $0$-bordant. This completes the proof of Lemma~\ref{l2}. 
 \end{proof}
 Now 
 $(f',b')$
is a normal degree one map
between manifolds, so it defines an element
$[f',b']\in H_{n}(M^{n}_{0};\mathbb{L}^{+})$, namely its normal invariant. 

\begin{corollary}\label{c2}
Consider the homomorphism
$
(f_{0})_{*}:
H_{n}(M^{n}_{0};\mathbb{L}^{+})
\to
H^{st}_{n}(X^{n};\mathbb{L}^{+}).
$  
Then
$
(f_{0})_{*}([f',b'])=[f,b]-[f_{0},b_{0}].
$
\end{corollary}
\begin{remark}
If $X^{n}$ happens to be a topological $n$-manifold, then this is the Ranicki composition formula (cf. Ranicki~\cite[Proposition 2.7]{Ra09}).
\end{remark}
\begin{proof}
The assertion follows from the following diagram
{\small
 \begin{equation}\label{d9}
\begin{tikzpicture}[auto, node distance=7cm]
  \node (A) {$H^{lf}_{n+2}(M^{n}_{0}\times \mathbb{R}_{+};\mathbb{\Omega}^{NSTOP})$};
  \node (B) [right of=A] {$H^{lf}_{n+2}(F_{0};\mathbb{\Omega}^{NSTOP})$};
  \node (A1) [below of=A,yshift=4.5cm] {$H^{st}_{n+2}(M^{n}_{0}\times [0,\infty], M^{n}_{0}\times \{\infty \}, M^{n}_{0}\times \{0\};\mathbb{\Omega}^{NSTOP})$};
  \node (B1) [below of=B, yshift=4.5cm] {$H^{st}_{n+2}(F, X^{n}, \{c_{0}\};\mathbb{\Omega}^{NSTOP})$};
  \node (A2) [below of=A1, yshift=4.5cm] {$H_{n+1}(M^{n}_{0};\mathbb{\Omega}^{NSTOP})$};
  \node (B2) [below of=B1, yshift=4.5cm] {$H^{st}_{n+1}(X^{n};\mathbb{\Omega}^{NSTOP})$};
  \node (A3) [below of=A2, yshift=4.5cm] {$H_{n}(M^{n}_{0};\mathbb{L}^{+})$};
  \node (B3) [below of=B2, yshift=4.5cm] {$H^{st}_{n}(X^{n};\mathbb{L}^{+})$};
  \draw[->, font=\small] (A)  to node  [midway, above] {$(\Gamma_{0})_{*}$} (B);
  \draw[shorten >=0.5cm,shorten <=0.5cm,->, font=\small] (A)  to node {$\cong$} (A1);
  \draw[shorten >=0.5cm,shorten <=0.5cm,->, font=\small] (B)  to node  {$\cong$} (B1);
  \draw[->, font=\small] (A1) to node [midway, above]{$(\bar{\Gamma}_{0})_{*}$} (B1);  
  \draw[shorten >=0.5cm,shorten <=0.5cm,->, font=\small] (A1) to node {} (A2);
  \draw[shorten >=0.5cm,shorten <=0.5cm,->, font=\small] (B1) to node {} (B2);
  \draw[->, font=\small] (A2) to node [midway, above]{$(f_{0})_{*}$} (B2);
  \draw[shorten >=0.5cm,shorten <=0.5cm,->, font=\small] (A2) to node {$(sign^{{\mathbb{L}}^{+}})_{*}$} (A3);
  \draw[shorten >=0.5cm,shorten <=0.5cm,->, font=\small] (B2) to node  {$(sign^{{\mathbb{L}}^{+}})_{*}$} (B3);
  \draw[->, font=\small] (A3) to node [midway, above]{$(f_{0})_{*}$} (B3);
\end{tikzpicture}
\end{equation}
}
Note that the element
$
[M((f',b')\times Id_{\mathbb{R}_{+}})]
\in
H^{lf}_{n+2}(M^{n}_{0}\times \mathbb{R}_{+}; \mathbb{\Omega}^{NSTOP})
$
maps to
$[f',b']$
under the left vertical arrow of morphisms.
The completion
of
$\Gamma_{0}$
then gives the map
$\bar{\Gamma}_{0}: M^{n}_{0} \times [0,\infty] \to F.$
This completes the proof of Corollary~\ref{c2}.
\end{proof}

 \subsection*{III}
 {\bf Summary:}
Let $X^{n}$ be an oriented closed generalized manifold of dimension $n\ge 5,$
and
$f_{0}:M^{n}_{0} \to X^{n}, \   b_{0}:\nu_{M^{n}_{0}}\to \xi_{0}$
a surgery problem according to a $BSTOP$-reduction of $\nu_{X^{n}}.$
Then the map 
$t:\mathcal{N}(X^{n}) \to H^{st}_{n} ( X^{n}; \mathbb{L}^+),$
defined in Section~\ref{s2}, fits into the
following commutative diagram 
  \begin{equation}\label{d10}
\begin{tikzpicture}[baseline=-0.8cm,  auto, node distance=2.5cm]
  \node (1) {$\mathcal{N}(M^{n}_{0})$};
  \node (2) [right of=1] {$ H_{n} ( M^{n}_{0}; \mathbb{L}^+)$}; 
  \node (11) [below of=1, yshift=1.0cm] {$\mathcal{N}(X^{n})$};
  \node (22) [below of=2, yshift=1.0cm] {$ H^{st}_{n} ( X^{n}; \mathbb{L}^+)$}; 
\draw[->, font=\small] (1) to node {$t_{0}$} (2); 
\draw[->, font=\small] (11) to node {} (1);
\draw[->, font=\small] (2) to node {$(f_{0})_{*}$} (22); 
\draw[->, font=\small] (11) to node {$t$} (22);   
\end{tikzpicture}
\end{equation}

 Here, $t_{0}$ sends a normal degree one map with target $M^{n}_{0}$ to its normal invariant.
 Moreover, under the identification of Corollary~\ref{c1}, diagram~\ref{d10} can be redrawn as follows  
 \begin{equation}\label{d11}
\begin{tikzpicture}[baseline=-1.0cm,  auto, node distance=3cm]
  \node (1) {$[M^{n}_{0},G/TOP]$};
  \node (2) [right of=1] {$H_{n} ( M^{n}_{0}; \mathbb{L}^+)$};
  \node (11) [below of=1, yshift=1.5cm] {$[X^{n},G/TOP]$};
  \node (22) [below of=2, yshift=1.5cm] {$ H^{st}_{n} ( X^{n}; \mathbb{L}^+)$};
\draw[->, font=\small] (1) to node {} (2);
\draw[->, font=\small] (11) to node {$(f_{0})^{*}$} (1);
\draw[->, font=\small] (2) to node {$(f_{0})_{*}$} (22);
\draw[->, font=\small] (11) to node {} (22);  
\end{tikzpicture}
\end{equation}
\section{Proof of Theorem~\ref{B}}\label{s4}  
 
The essence of the proof was already given in Section~\ref{s1},
by comparing diagrams \ref{d1} and \ref{d2}.
In this section  we present the details.
 
Let $\mathbb{L}^{\bullet}$ denote the symmetric 
  $\mathbb{L}$-spectrum (cf. Ranicki~\cite[Chapter 13]{Ra92}). It is a ring spectrum and 
   $\mathbb{L}^{+}$ 
   is a
   $\mathbb{L}^{\bullet}$-module spectrum.
   Hence the cup product constructions
   $
   H^{q}(Z,A; \mathbb{L}^{\bullet})
   \times
   H^{p}(Z; \mathbb{L}^{+})
   \to
   H^{p+q}(Z,A; \mathbb{L}^{+})
   $
   are well-defined. 
   
 Considering
 an oriented closed generalized $n$-manifold, embedded in $ X^{n} \subset S^{m},$
 for some $m \ge n,$
  its Spivak fibration $\nu_{X^{n}}$
 has a canonical orientation
 (cf. Ranicki~\cite[Chapter 16]{Ra92}),
   i.e. a Thom class
    \[
    \mathcal{U}_{\nu_{X^{n}}}
    \in
     H^{m-n}(E(\nu_{X^{n}}), \partial E(\nu_{X^{n}}); \mathbb{L}^{\bullet})
     \cong
     \tilde{H}^{m-n}(T(\nu_{X^{n}}); \mathbb{L}^{\bullet}),
    \] 
     inducing the Thom isomorphism (here,
     $E(\nu_{X^{n}})$
      is the associated disk fibration)
     \[
      H^{0}(X^{n}; \mathbb{L}^{+})=
      H^{0}(E(\nu_{X^{n}}); \mathbb{L}^{+})
      \xrightarrow[]{\cup \  \mathcal{U}_{\nu_{X^{n}}}} 
       \tilde{H}^{m-n}(T(\nu_{X^{n}}); \mathbb{L}^{+}).
     \] 
     
  Recall that {\it canonical} means that it is constructed via the canonical reduction $\xi_{0}$
   of $\nu_{X^{n}}.$
   Hence the Thom class
      $
    \mathcal{U}_{\xi_{0}}
    \in
     \tilde{H}^{m-n}(T(\xi_{0}); \mathbb{L}^{\bullet}),
     $
      corresponds to 
      $ \mathcal{U}_{\nu_{X^{n}}} $
      under the homotopy equivalence between
      $ T(\xi_{0})$
      and
      $T(\nu_{X^{n}}). $
      
  The existence of 
      $  \mathcal{U}_{\xi_{0}} $
  is guaranteed   (cf. Ranicki~\cite[Chapter 16]{Ra92}).
  Moreover, since 
      $(f_{0})^{*}(\xi_{0})
      \cong
       \nu_{M^{n}_{0}},$
    it follows that 
     $
   f_{0}:M^{n}_{0} \to X^{n}, \  b_{0}:\nu_{M^{n}_{0}}\to \xi_{0}
     $
   induces
      $T(b_{0}):T(\nu_{M^{n}_{0}}) \to T(\xi_{0}),$
      so that under 
     \[
      (T(b_{0}))^{*}:H^{m-n}(T(\xi_{0});  \mathbb{L}^{\bullet}) 
      \to
      H^{m-n}(T(\nu_{M^{n}_{0}});   \mathbb{L}^{\bullet}),
     \] 
      $\mathcal{U}_{\xi_{0}}$
      is mapped to
      $\mathcal{U}_{\nu_{M^{n}_{0}}}$, 
      the Thom class of
      $\nu_{M^{n}_{0}}.$
 This implies commutativity of the following diagram 
 \begin{equation}\label{d12}
\begin{tikzpicture}[baseline=-1.0cm,  auto, node distance=4.5cm]
  \node (1) {$H^{0}(M^{n}_{0}; \mathbb{L}^{+})$};
  \node (2) [right of=1] {$\tilde{H}^{m-n}(T(\nu_{M^{n}_{0}});   \mathbb{L}^{+})$}; 
  \node (11) [below of=1, yshift=2.5cm] {$H^{0}(X^{n}; \mathbb{L}^{+})$};
  \node (22) [below of=2, yshift=2.5cm] {$\tilde{H}^{m-n}(T(\xi_{0});   \mathbb{L}^{+})$}; 
\draw[->, font=\small] (1) to node {$. \cup   \mathcal{U}_{\nu_{M^{n}_{0}}}$} (2); 
\draw[->, font=\small] (11) to node {$(f_{0})^{*}$} (1);
\draw[->, font=\small] (22) to node {$(T(b_{0}))^{*}$} (2); 
\draw[->, font=\small] (11) to node {$. \cup  \mathcal{U}_{\nu_{X^{n}}}$} (22);   
\end{tikzpicture}
\end{equation} 

 The Thom isomorphisms are now composed with the $S$-duality isomorphisms:
 \[
 \tilde{H}^{m-n}(T(\nu_{M^{n}_{0}});   \mathbb{L}^{+})
 \cong
 H^{st}_{n}(M^{n}_{0};  \mathbb{L}^{+})
 \cong
  H_{n}(M^{n}_{0};  \mathbb{L}^{+})
 \]   
 and
\[
 \tilde{H}^{m-n}(T(\nu_{X^{n}});   \mathbb{L}^{+})
 \cong
 H^{st}_{n}(X^{n};  \mathbb{L}^{+}).
 \] 
 
 For the generalized manifold $X^{n},$
 this follows from  Kahn, Kaminker and Schochet~\cite[Theorem B]{KKS77}, which asserts that
 \[
 H^{m-n-1}(S^{m}\setminus X^{n};\mathbb{L}^{+} )
 \cong
 H^{st}_{n}(X^{n}; \mathbb{L}^{+}).
 \]  
Since for every $m \ge n,$
 $$H^{m-n-1}(S^{m};\mathbb{L}^{+})=L_{m-1}, \ \ 
 H^{m-n}(S^{m};\mathbb{L}^{+})=L_{m},$$
 where $L_{q}=\pi_{q}(G/TOP),$
 the exact sequence of the pair $(S^{m},S^{m}\setminus X^{n})$ then implies that 
 \[  
 H^{m-n-1}(S^{m}\setminus X^{n};\mathbb{L}^{+})
 \cong 
 \tilde{H}^{m-n}(S^{m},S^{m}\setminus X^{n};\mathbb{L}^{+})
 \]
 \[
 \cong 
  \tilde{H}^{m-n}(T(\nu_{X^{n}});\mathbb{L}^{+} )  \cong   \tilde{H}^{m-n}(T(\xi_{0});\mathbb{L}^{+} ).  \]    
 This also applies to $M^{n}_{0}.$
 
 The proof of Kahn, Kaminker and Schochet~\cite[Theorem B]{KKS77}
 shows that  the following diagram commutes 
 \begin{equation}\label{d13}
\begin{tikzpicture}[baseline=-1.0cm,  auto, node distance=4.5cm]
  \node (1) {$\tilde{H}^{m-n}(T(\nu_{M^{n}_{0}});\mathbb{L}^{+} )$};
  \node (2) [right of=1] {$H^{st}_{n}(M^{n}_{0};  \mathbb{L}^{+})=H_{n}(M^{n}_{0};  \mathbb{L}^{+})$};
  \node (11) [below of=1, yshift=2.5cm] {$ \tilde{H}^{m-n}(T(\xi_{0});\mathbb{L}^{+} )$};
  \node (22) [below of=2, yshift=2.5cm] {$H^{st}_{n}(X^{n};  \mathbb{L}^{+})$};
\draw[->, font=\small] (1) to node {$\cong$} (2);
\draw[->, font=\small] (11) to node {$(T(b_{0}))^{*}$} (1);
\draw[->, font=\small] (2) to node {$(f_{0})_{*}$} (22);
\draw[->, font=\small] (11) to node {$\cong$} (22);  
\end{tikzpicture}
\end{equation} 

Briefly, this follows since the Spanier-Whitehead duality isomorphism comes from the slant product constructions, using the map
\[
X^{n}_{+} \wedge T(\xi_{0})
\cong 
X^{n}_{+} \wedge T(\nu_{X^{n}})
\to 
S^{m},
\]   
i.e. it comes from the element in 
$
H^{m}(X^{n}_{+} \wedge T(\xi_{0});\mathbb{L}^{+})
$
which it defines.
This construction is natural for the normal map $(f_{0},b_{0}).$
Since $X^{n}$ is not a complex, $T(\nu_{X^{n}})$
is replaced by a certain function space which leads to the Steenrod homology 
(cf.  Kahn, Kaminker and Schochet~\cite[Section 4]{KKS77}).
\smallskip

 {\bf Summary:} The following diagram commutes
\begin{equation}\label{d14}
\begin{tikzpicture}[baseline=-1cm, auto, node distance=3.5cm]
  \node  (1) {$H^{0}(M^{n}_{0};\mathbb{L}^{+})$};
  \node   (3) [right of=1] {$\tilde{H}^{m-n}(T(\nu_{M^{n}_{0}});\mathbb{L}^{+})$};
  \node  (4) [right of=3] {$H_{n}(M^{n}_{0};\mathbb{L}^{+})$};   
  \node   (11) [below of=1, yshift=1.5cm] {$H^{0}(X^{n};\mathbb{L}^{+})$};
  \node   (33) [below of=3, yshift=1.5cm] {$\tilde{H}^{m-n}(T(\xi_{0});\mathbb{L}^{+})$};   
  \node  (44) [below of=4, yshift=1.5cm] {$H^{st}_{n}(X^{n};\mathbb{L}^{+})$};   
\draw[->, font=\small] (1) to node {$\cong$} (3); 
\draw[->, font=\small] (3) to node {$SD$} (4);
\draw[->, font=\small] (11) to node {$(f_{0})^{*}$} (1); 
\draw[->, font=\small] (33) to node {${(T(b_{0}))^{*}}$} (3);
\draw[->, font=\small] (4) to node [swap] {$(f_{0})_{*}$} (44);
\draw[->, font=\small] (11) to node {$\cong$} (33);   
\draw[->, font=\small] (33) to node {$SD$} (44);    
\end{tikzpicture}
\end{equation} 
The composition of the upper row isomorphisms is known to be
\[
. \cap [M^{n}_{0}]_{\mathbb{L}^{\bullet}}:
H^{0}(M^{n}_{0};\mathbb{L}^{+})
\xrightarrow[]{\cong}
H_{n}(M^{n}_{0};\mathbb{L}^{+}),
\]
where
$[M^{n}_{0}]_{\mathbb{L}^{\bullet}}
\in
H_{n}(M^{n}_{0};\mathbb{L}^{\bullet})$
is the 
$\mathbb{L}^{\bullet}$-coefficient fundamental class of $M^{n}_{0}$
 (cf. Ranicki~\cite[Proposition 18.3]{Ra92}).
Finally, we can identify 
\[
[M^{n}_{0},G/TOP]=
H^{0}(M^{n}_{0};\mathbb{L}^{+}),
\ \
[X^{n},G/TOP]=
H^{0}(X^{n};\mathbb{L}^{+} ),
\]
according to the equivalence
$
G/TOP \xrightarrow[]{\cong} \mathbb{L}^{+}
$
 (cf. Kirby and Siebenmann~\cite[Essay 5, Appendix C]{KiSi77}, Ranicki~\cite[Proposition 16.1]{Ra92}).
 
Combining this with Corollary~\ref{c1} and diagram~\ref{d10}, we obtain  the following diagram
\begin{equation}\label{d15}
\begin{tikzpicture}[baseline=-1.0cm, ->,>=stealth', auto, node distance=3cm, main node/.style={font=\small}]
  \node [main node] (1) {$H^{0}(M^{n}_{0};\mathbb{L}^{+})$};
  \node [main node] (2) [right of=1] {$[M^{n}_{0},G/TOP]$};
  \node [main node] (3) [right of=2] {$\mathcal{N}(M^{n}_{0})$};
  \node [main node] (4) [right of=3] {$H_{n}(M^{n}_{0};\mathbb{L}^{+})$};    
 \path[every node/.style={font=\sffamily\small}]
    (1) edge node [right] {} (2)
    (2) edge node [right] {} (3)
    (3) edge node [right] {} (4)
    (1) edge[bend left] node [left,above] {$. \cap [M^{n}_{0}]_{\mathbb{L}^{+}}$} (4);          
  \node [main node] (11) [below of=1, yshift=1.0cm] {$H^{0}(X^{n};\mathbb{L}^{+})$};
  \node [main node] (22) [below of=2, yshift=1.0cm] {$[X^{n},G/TOP]$};
  \node [main node] (33) [below of=3, yshift=1.0cm] {$\mathcal{N}(X^{n})$};
  \node [main node] (44) [below of=4, yshift=1.0cm] {$H^{st}_{n}(X^{n};\mathbb{L}^{+})$};  
\path[every node/.style={font=\small}]  
    (11) edge node [right] {} (22)
    (22) edge node [right] {} (33)
    (33) edge node [right] {} (44)
    (11) edge[bend right] node [left, above] {$\cong$} (44);       
\draw[->, font=\small] (1) to node {$\cong$} (2);
\draw[->, font=\small] (2) to node {$\cong$} (3);
\draw[->, font=\small] (3) to node {$t_{0}$} (4);
\draw[->, font=\small] (11) to node {$(f_{0})^{*}$} (1);
\draw[->, font=\small] (22) to node {$(f_{0})^{*}$} (2);
\draw[->, font=\small] (33) to node {} (3);
\draw[->, font=\small] (4) to node [swap] {$(f_{0})_{*}$} (44);
\draw[->, font=\small] (11) to node {$\cong$} (22);  
\draw[->, font=\small] (22) to node {$\cong$} (33);
\draw[->, font=\small] (33) to node {$t$} (44);    
\end{tikzpicture}
\end{equation}

Commutativity of the outer diagram (cf. diagram~\ref{d14}) and  each square imply that
\[
\mathcal{N}(X^{n}) 
\to
 \mathcal{N}(M^{n}_{0})
\xrightarrow[]{t_{0}}
H_{n}(M^{n}_{0};\mathbb{L}^{+})
\xrightarrow[]{(f_{0})_{*}}
H^{st}_{n}(X^{n};\mathbb{L}^{+})
\]
is an isomorphism, hence by diagram~\ref{d10}, this composition is $t.$
This completes the proof of Theorem~\ref{B}.\qed
\begin{remark}
 In particular, the proof of
 Theorem~\ref{B}
  also  shows
   that the $\mathbb{L}$-duality isomorphism for generalized
manifold $X^{n}$ factors over $t:\mathcal{N}(X^{n}) \to H^{st}_{n} ( X^{n}; \mathbb{L}^+)$.
\end{remark}
\section{Epilogue}\label{s5}

We shall conclude this paper by a brief outlook for further studies, following a very interesting suggestion of the referee. In this paper we have proved that there exists a bijective map
 $t:\mathcal{N}(X^{n}) \to H^{st}_{n} ( X^{n}; \mathbb{L}^{+})$
 from normal degree one bordisms
 to the Steenrod homology of the spectrum
 $\mathbb{L}^{+}$.
 
 The Steenrod homology is known to behave well on
 the category of compact metric spaces.
 In particular, if
 $q:X^{n}\to B$
 is any morphism,
 then there exists a long exact sequence
 \begin{equation}\label{eq1}
\dots
\to 
H^{st}_{n+1}(B;\mathbb{L}^{+})
\to
H^{st}_{n+1}(B,X^{n};\mathbb{L}^{+})
\xrightarrow[]{\partial_{*}}
H^{st}_{n}(X^{n};\mathbb{L}^{+})
\xrightarrow[]{q_{*}}
H^{st}_{n}(B;\mathbb{L}^{+})
\to
\dots
\end{equation}
On the other hand, if
 $q:X^{n}\to B$
 is a $UV^{1}$-map,
 then
 there is a controlled surgery sequence
 (cf. 
 Bryant, Ferry, Mio and Weinberger~\cite{BFMW96},
 Mio~\cite{Mi00},
 and Nicas~\cite{Ni82}),
\begin{equation}\label{eq2}
H^{st}_{n+1}(B;
\mathbb{L})
\to 
{\mathcal{S}^{c}}
\Biggl(
\CD
X^{n}\\
@VVV{q}\\
B\\
\endCD
\Biggr)
\xrightarrow[]{}
\mathcal{N}(X^{n})
\xrightarrow[]{\sigma^{c}}
H^{st}_{n}(B;\mathbb{L})
\end{equation}

Here, $\mathbb L$ denotes the $4$-periodic spectrum with  
$\mathbb{L}_{0}=\mathbb{Z} \times G/TOP$,
$\sigma^{c}$ is the controlled surgery obstruction map, and 
\begin{equation}\label{set}
{\mathcal{S}^{c}}
\Biggl(
\CD
X^{n}\\
@VVV{q}\\
B\\
\endCD
\Biggr)
\end{equation}
is the controlled structure set.
This  controlled surgery sequence~\ref{eq2}
makes sense if the controlled structure set \ref{set} is nonempty
(cf. Mio~\cite[Theorem 3.8]{Mi00}). 

It is natural to ask if sequences
\ref{eq1} and \ref{eq2} are related via the map
 $t:\mathcal{N}(X^{n}) \to H^{st}_{n} ( X^{n}; \mathbb{L}^{+})$. 
 First, one notes that two spectra 
 $\mathbb{L}^{+}$
 and
 $\mathbb{L}$
 are involved, where
 $\mathbb{L}^{+}
 \xrightarrow[]{i}
\mathbb{L}$
is considered as the covering spectrum over the Eilenberg-MacLane spectrum
$K(\mathbb{Z},0)$, i.e.
 $\mathbb{L}^{+}
 \xrightarrow[]{i}
\mathbb{L}\to 
K(\mathbb{Z},0)
$
is a fibration of spectra.

In order to compare sequences
\ref{eq1} and \ref{eq2}  we consider the composite map
$$
q_{*}\circ i_{*}:
H^{st}_{n}(X^{n};\mathbb{L}^{+})
\xrightarrow[]{i_{*}}
H^{st}_{n}(X^{n};\mathbb{L})
\xrightarrow[]{q_{*}}
H^{st}_{n}(B;\mathbb{L}),
$$
and
obtain the following diagram
 \begin{equation}\label{d16}
\begin{tikzpicture}[baseline=-0.8cm,  auto, node distance=3cm]
  \node (1) {${H}^{st}_{n}(X^{n};\mathbb{L}^{+})$};
  \node (2) [right of=1] {${H}^{st}_{n}(B;\mathbb{L})$};
  \node (11) [below of=1, yshift=1.5cm] {$\mathcal{N}(X^{n})$};
  \node (22) [below of=2, yshift=1.5cm] {${H}^{st}_{n}(B;\mathbb{L})$};
\draw[->, font=\small] (1) to node {$q_{*}\circ i_{*}$} (2);
\draw[->, font=\small] (11) to node {$t$} (1);
\draw[->, font=\small] (22) to node {=} (2);
\draw[->, font=\small] (11) to node {$\sigma^{c}$} (22);  
\end{tikzpicture}
\end{equation} 
 The first step
 would be to prove commutativity of diagram \ref{d16}.
 However, this is not enough, since one also needs a map 
 between 
 ${H}^{st}_{n+1}(B, X^{n};\mathbb{L}^{+})$
 and the set \ref{set},
 compatible with 
 $t:\mathcal{N}(X^{n}) \to H^{st}_{n} ( X^{n}; \mathbb{L}^{+})$. 
 This can all be done if $X^{n}$
 is a topological $n$-manifold
 (cf. Hegenbarth and Repov\v{s}~\cite{HeRe06}).
  In the case when $X^{n}$ is a generalized $n$-manifold, this is still an unsolved problem.
  
For the second step, one is led to ''refining''
  the map
 $t:\mathcal{N}(X^{n}) \to H^{st}_{n} ( X^{n}; \mathbb{L}^{+})$
 to a map 
 $$
 \overline{t}:
 {\mathcal{S}^{c}}
\Biggl(
\CD
X^{n}\\
@VVV{q}\\
B\\
\endCD
\Biggr)
\to H^{st}_{n+1} (B, X^{n}; \mathbb{L}^{+})$$
so that the following diagram is commutative
\begin{equation}\label{d17}
\begin{tikzpicture}[baseline=-1cm, ->,>=stealth', auto, node distance=3cm, main node/.style={font=\small}]
  \node [main node] (1) {$H^{st}_{n+1}(B;\mathbb{L}^{+})$};
  \node [main node] (2) [right of=1] {$H^{st}_{n+1}(B, X^{n};\mathbb{L}^{+})$};
  \node [main node] (3) [right of=2] {$H^{st}_{n}(X^{n};\mathbb{L}^{+})$};
  \node [main node] (4) [right of=3] {$H^{st}_{n}(B;\mathbb{L})$};    
  \node [main node] (11) [below of=1, yshift=1.5cm] {$H^{st}_{n+1}(B;\mathbb{L})$};
  \node [main node] (22) [below of=2, yshift=1.5cm] {${\mathcal{S}^{c}}$};
  \node [main node] (33) [below of=3, yshift=1.5cm] {$\mathcal{N}(X^{n})$};
  \node [main node] (44) [below of=4, yshift=1.5cm] {$H^{st}_{n}(B;\mathbb{L})$};  
\draw[->, font=\small] (1) to node {} (2);
\draw[->, font=\small] (2) to node {} (3);
\draw[->, font=\small] (3) to node {$q_{*}\circ i_{*}$} (4);
\draw[->, font=\small] (1) to node {$i_{*}$} (11);
\draw[->, font=\small] (22) to node {$\overline{t}$} (2);
\draw[->, font=\small] (33) to node {$t$} (3);
\draw[->, font=\small] (4) to node [swap] {=} (44);
\draw[->, font=\small] (11) to node {} (22);  
\draw[->, font=\small] (22) to node {} (33);
\draw[->, font=\small] (33) to node {$\sigma^{c}$} (44);    
\end{tikzpicture}
\end{equation}
where 
${\mathcal{S}^{c}}$ 
denotes the set \ref{set}.

Since $\dim X^{n}=n,$ we may assume that $\dim B\leq n.$
In this case, it follows from the Atiyah-Hirzebruch spectral sequence
(which holds for the Steenrod homology, cf. Hegenbarth and Repov\v{s}~\cite[p. 206]{HeRe20})
that 
$H^{st}_{n+1}(B;\mathbb{L}^{+})
\xrightarrow[]{i_{*}}
H^{st}_{n+1}(B;\mathbb{L})$
is an isomorphism.
In this case, the map
$$
 \overline{t}:
 {\mathcal{S}^{c}}
\Biggl(
\CD
X^{n}\\
@VVV{q}\\
B\\
\endCD
\Biggr)
\to H^{st}_{n+1} (B, X^{n}; \mathbb{L}^{+})$$
is bijective.
However, the existence of such a map
$\overline{t}$
is at present still a conjecture.

\subsection*{Acknowledgements}
We gratefully acknowledge John Bryant for his remark regarding Proposition~\ref{p}.
 We also thank the referee for several comments and suggestions.
This research was supported by the Slovenian Research Agency grants P1-0292, N1-0083, and N1-0114.

\end{document}